\documentclass[a4paper, 11pt]{article}

\usepackage{calc, amsmath, amsthm, a4, latexsym, amssymb, color, url, soul}
\usepackage{graphicx}
\definecolor{comcolor}{rgb}{0.9,0.3,0.3}
\definecolor{starcolor}{rgb}{0.3,0.3,0.9}
\definecolor{hscolor}{rgb}{0.9,0.6,0.5}
\definecolor{darkgreen}{rgb}{0.1,0.6,0.3}

\newtheorem{thm}{Theorem}[section]
\newtheorem{lemma}[thm]{Lemma}
\newtheorem{prop}[thm]{Proposition}

\theoremstyle{definition}
\newtheorem{defn}[thm]{Definition}

\newtheorem{rem}[thm]{Remark}

\newcommand{\be}[1]{\begin{equation}\label{#1}}
\newcommand{\ee}{\end{equation}}
\newcommand{\ba}{\begin{array}}
\newcommand{\ea}{\end{array}}
\newcommand{\bal}{\begin{aligned}}
\newcommand{\eal}{\end{aligned}}

\newcommand{\R}{\mathbb{R}}

\newcommand{\N}{\mathbb{N}}

\newcommand{\E}{\mathbb{E}}

\renewcommand{\P}{\mathbb{P}}

\newcommand{\1}{1\hspace{-0.098cm}\mathrm{l}}

\newcommand{\dd}{{\text{d}}}

\begin{document}

\begin{center}
{\Large \bf 
Separation of time-scales for the seed bank diffusion and its jump-diffusion limit
}\\[5mm]

\vspace{0.7cm}
\textsc{Jochen Blath$^1$, Eugenio Buzzoni$^1$ , Adri\'an Gonz\'alez Casanova$^2$ , Maite Wilke Berenguer$^3$} 

\vspace{0.5cm}
$^1$ Institut f\"ur Mathematik, Technische Universit\"at Berlin, Germany.\\$^2$ Instituto de Matem\'aticas, Universidad Nacional Aut\'onoma de M\'exico,  Mexico.\\
$^3$ Fakult\"at f\"ur Mathematik, Ruhr-Universit\"at Bochum, Germany
\end{center}

\begin{abstract}
\noindent We investigate the scaling limit of the seed bank diffusion when reproduction and migration (to and from the seed bank) happen on different time-scales. More precisely, we consider the case when migration is `slow' and reproduction is `standard' (in the original time-scale) and then switch to a new, accelerated time-scale, where migration is `standard' and reproduction is `fast'. This is motivated by models for bacterial dormancy, where periods of quiescence can be orders of magnitude larger than reproductive times, and where it is expected to find non-trivial degenerate genealogies on the evolutionary time-scale.

However, the above scaling regime is not only interesting from a biological perspective, but also from a mathematical point of view, since it provides a prototypical example where the expected scaling limit of a continuous diffusion should (and will be) a jump-diffusion. For this situation, standard convergence results often seem to fail in multiple ways. For example, since the set of continuous paths from a closed subset of the c\`adl\`ag paths in each of the classical Skorohod topologies $J_1, J_2, M_1$ and $M_2$, none of them can be employed for tightness on path-space. Further, a na\"ive direct rescaling of the Markov generator corresponding to the continuous diffusion immediately leads to a blow-up of the diffusion coefficient. Still, one can identify a well-defined limit via duality in a surprisingly non-technical way. Indeed, we show that a certain duality relation is in some sense stable under passage to the limit and allows an identification of the limit, avoiding all technicalities related to the blow-up in the classical generator. The result then boils down to a convergence criterion for time-continuous Markov chains in a separation of time-scales regime, which is of independent interest.

  \par\medskip
\end{abstract}


\section{Overview and main results}
\label{sec:intro}

Seed bank models have drawn considerable interest in population genetics recently, and, reflecting the complexity of modeling evolutionary processes in biology, there is now a rather large variety of concrete ways to model seed banks in population genetics in the literature, see e.g.\ \cite{SL18} and \cite{BK19} for overviews and further references. For example, many microbial species are able to enter a reversible state of vanishing metabolic activity, and it seems that at any given time, a large fraction of microorganism in nature is in such a dormant state (\cite{LJ11}). One of the most natural ways to model a corresponding seed bank in population genetics is to treat the switching between active and dormant states as `migration' between two `islands' (the active and the dormant population) in the spirit of Wright's classical two island model (\cite{W31, M59}), with the notable difference that reproduction is turned off in the dormant island. In a bi-allelic population (say with types $\{a, A\})$ following such a Wright-Fisher model with two islands, say with $N$ active individuals and $K\cdot N$ dormant individuals (for some suitable positive constant $K$), and individual's switching probability of size $c/N$ in each generation (for some $c>0$), after letting the population size $N \to \infty$ and rescaling time in the classical way by $N$, one obtains the so-called {\em seed bank diffusion} introduced in \cite{BGCKW16}:

\begin{defn}[Seed bank diffusion]
\label{defn:system}
Let $(B_t)_{t\geq 0}$ be a standard Brownian motion and $c, K$ finite positive constants. 
Consider the $[0,1]^2$-valued continuous strong Markov process $(X(t), Y(t))_{t \geq 0}$ that is the unique strong solution of the initial value problem 
\begin{align}
\label{eq:system}
\text{d} X(t) & = c(Y(t) -X(t))\text{d}t + \sqrt{X(t)(1-X(t))}\text{d}B_t, \notag \\[.1cm]
\text{d} Y(t) & =  Kc(X(t) -Y(t))\text{d}t,
\end{align}
with $(X(0), Y(0)) =(x,y) \in [0,1]^2$. Then, $(X(t), Y(t))_{t \geq 0}$ is called the {\em Wright-Fisher diffusion with seed bank} with parameters $c,K$, starting in $(x, y)\in [0,1]^2$.
\end{defn}

Here, we interpret $X(t)$ as the frequency of active individuals of type $a$ and $Y(t)$ as the frequency of dormant individuals of type $a$ at time $t$. 
The underlying scaling assumptions of the above model imply that the time that an individual spends in the dormant population is of order $N$ (see \cite{BEGCKW15} for a detailed discussion of the scaling assumptions). 

However, it has been reported that in various bacterial species, single individuals may stay inactive for extremely long periods of time, several orders of magnitude longer than the reproductive time (\cite{ Ancient_Bacteria_2,Ancient_Bacteria}). Further, it can be expected that in such a scenario, one observes non-classical behavior of the genealogy over long time-scales (e.g., `extinct' types my be reintroduced after long time periods, though this should happen rarely). These considerations motivate the investigation of the scaling limit of the above system when migration between active and dormant states (rate $c$) and reproduction (rate 1) happen on different time-scales. Interesting limits can only be expected when switching to a faster time-scale. Indeed, if one just lets $c \to 0$, then one obtains the trivial (and uninteresting) limit where the active population is completely separate from the dormant population and simply follows a classical Wright-Fisher diffusion. Hence we speed up time by a factor $ 1/c$, as $c \to 0$, and transition to the new  time-scale, where now migration between states happens at rate 1 while reproduction happens `instantaneously'. While such a separation of time-scales can be expected to lead to an interesting process, the na\"ive scaling limit 
\begin{align}
\text{d} X(t) & = (Y(t) -X(t))\text{d}t + ``\infty" \sqrt{X(t)(1-X(t))}\text{d}B_t, \notag \\[.1cm]
\text{d} Y(t) & = K(X(t) -Y(t))\text{d}t,
\end{align}
of course does not make sense (observe the $``\infty"$) in front of the diffusion coefficient due to speeding up time).

Intuitively, fast reproduction should drive the process immediately towards the boundaries, and only rarely should one switch from 0 to 1 or vice versa due to immigration. Yet, it is not completely obvious how to make this idea rigorous. Fortunately, the seed bank diffusion has a nice moment dual, which we can use for further investigations: 

\begin{defn}[Block-counting process of the seed bank coalescent]  
\label{defn:blockcountingSBC}
Let $E := {\N}_0\times {\N}_0 $ equipped with the discrete topology. Let $c,K>0$. We define $(N(t),M(t))_{t\geq 0}$ to be the time-continuous Markov chain with values in $E$ characterized by the rates $r_{(n,m), (\bar n,\bar m)} $ given by: 
\begin{align} \label{eq:ratesSBC}
r_{(n,m), (\bar n,\bar m)} = \begin{cases}
		   \binom{n}{2}& \text{if } (\bar n,\bar m) = (n-1,m,),\\
		    cn & \text{if } (\bar n,\bar m) = (n-1,m+1),\\
		    cKm  & \text{if } (\bar n,\bar m) = (n+1,m-1),
               \end{cases}
\end{align}
when $(\bar n,\bar m),\,(n,m)\in \N_0\times \N_0$ and zero otherwise off the diagonal. 
\end{defn}

In \cite{BGCKW16} it is shown that this continuous Markov chain process satisfies the moment duality
 \begin{equation}\label{eq:duality}
 \E_{(x,y)}\big[X_t^nY_t^m\big]= \E_{(n,m)}\big[x^{N_t}y^{M_t}\big]
 \end{equation}
 for every $t>0$, for every $(x,y)\in [0,1]$ and  for every $n,m\in\N_0$.

In other words, the distribution of the seed bank diffusion at any time $t$ is uniquely determined by the moment dual at said time. It is thus a natural idea to investigate the scaling limit of the moment dual under the same scaling assumption, which is potentially easier to get than the one for the original diffusion, and hope to obtain a well-defined limit which still provides information about the scaling limit of the original diffusion. Still, we encounter technical challenges since the limiting objects might not have standard semi-groups. Indeed, when speeding up time  in the time-continuous Markov chain, some transition rates diverge to $\infty$, thus obstructing direct $Q$-matrix computations and producing states that are vacated immediately. This phenomenon is frequently observed when dealing with ``separation-of-time-scales phenomena''(cf.\ for example \cite[Chapter 6]{W08}) and can in the best case scenario still lead to a scaling limit with ``degenerate'' (non-standard) transition semigroup of the form
\begin{align*}
Pe^{Gt}, \quad t \ge 0,
\end{align*}
where $P$ is a {\em projection} to a subspace of the original state space as a result of ``immediately vacated states'' and $G$ a ``classical'' conservative $Q$-matrix. For {\em time-discrete} Markov chains, this situation was considered e.g.\ in \cite{M98} and also \cite{BBE13}; see further \cite{MN16}. Since this might be of general interest, we give, a detailed ``recipe'' for such proofs for time-continuous Markov chains in Section \ref{subsec:Strategy_of_proof}. 

Subsequently, we apply this strategy to our model in Section \ref{subsec:ancestral_material_coal} and obtain the following results: Recall that we are interested in the (very) long term effect of a seed bank in which individuals change their state seldomly. Heuristically, this means that the switch between active and inactive takes a long time to happen as depicted in Figure \ref{fig:ancient_coalescent}. 
In this setting, the time that two active lines needs to find a common ancestor becomes negligible compared with the time that an ancestral line takes to change its state. The consequence of this is that, in the scaling limit, two active ancestral lines coalesce instantaneously while each line changes of state after a random time of order one, as described by the \emph{ancient ancestral lines process}:
\begin{figure}
\center
\includegraphics[width=.8\textwidth]{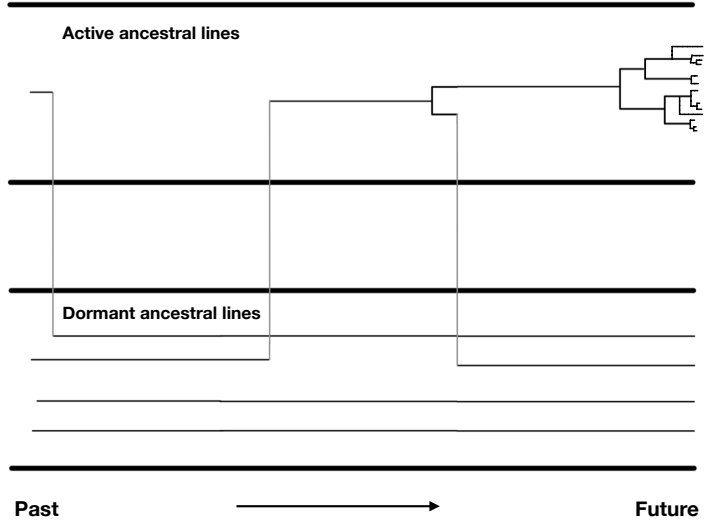}
\caption{For minuscule migration rate $c$, coalescent events occur rapidly in comparison to migration events as depicted above. As a result, passing to the limit on a scale that normalizes the migration rates, means coalescence events occur instantaneously and any randomness lies in the migration events.} \label{fig:ancient_coalescent}
\end{figure}
\begin{defn}[The ancient ancestral lines process]
\label{defn:ancgen}
Let $(n_0,m_0) \in \N_0 \times \N_0$. 
The \emph{$(n_0,m_0)$ ancient ancestral lines process} is the time-continuous Markov chain $(\tilde N(t), \tilde M(t))_{t \geq 0}$ with initial value $(\tilde N(0), \tilde M(0))=(n_0, m_0)$, taking values in the state space 
\begin{align*}
E_{(n_0,m_0)}:=\lbrace 0, \dots, n_0+m_0 \rbrace^2,
\end{align*}
with transition matrix
\begin{align*}
 \Pi(t):=Pe^{tG}, \qquad t>0
\end{align*}
and $\Pi(0)$ is the identity on $E$. $P$ is a projection given by 
\begin{align}\label{eq:Pancient}
 P_{(n,m),(\bar n,\bar m)} := \begin{cases}
                                1, 	& \text{ if } \bar n = 1,\, n \geq 1,\, \bar m = m,\\
                                1, 	& \text{ if } \bar n = n = 0,\, \bar m = m,\\
                                0,	& \text{ otherwise,}
                             \end{cases}
\end{align}
for all sensible $(n,m),\,(\bar n, \bar m) \in E_{(n_0.m_0)}$ and $G$ is a matrix of the form 

\begin{align*}
 G_{(n,m),(\bar n,\bar m)} := \begin{cases}
                              Km,	& \text{ if } \bar n = 1,\, n \geq 0,\, \bar m = m-1,\\
                              n,	& \text{ if } \bar n = 0,\, n \geq 1,\, \bar m = m+1,\\
                              -n-Km,	& \text{ if } \bar n = 1,\, n \geq 1,\, \bar m = m,\\
                              -Km,	& \text{ if } \bar n = n = 0,\, \bar m = m,\\
                              0,	& \text{ otherwise.}
                             \end{cases}
\end{align*}
\end{defn}
The projection acts for any $t>0$, hence this process `immediately' takes values in the smaller space $\{0,1\}\times\{0, \ldots, m_0+1\}$. The first two rates given in the definition of $G$ correspond to the events of resuscitation with immediate coalescence and dormancy. Note that $G$ is, however, not a $Q$-matrix: for any $\bar n \geq 2$ its negative values are off the diagonal. 

Using the techniques of Section \ref{subsec:Strategy_of_proof}, we prove that the \emph{ancient ancestral lines process} arises as the scaling limit of the block-counting process of the seed bank coalescent.

\begin{thm}
\label{thm:newcoal}%
Denote by $(N^c(t), M^c(t))_{t \geq 0}$ the block counting process of the seed bank coalescent as defined in Definition \ref{defn:blockcountingSBC} with migration rate $c>0$ and assume that it starts in  some $(n_0, m_0) \in \N\times\N$, $\P$-a.s.

Furthermore let $(\tilde N(t), \tilde M(t)))_{t \geq 0}$ be the ancient ancestral lines process from Definition \ref{defn:ancgen} with the same initial condition.  Then, 
for any sequence of migration rates $(c_{\kappa})_{\kappa \in \N}$ with $c_{\kappa} \rightarrow 0$ for $\kappa \rightarrow \infty$ 
\begin{align*}
\left(N^{c_{\kappa}}\left(\frac{1}{c_{\kappa}} t\right), M^{c_{\kappa}}\left(\frac{1}{c_{\kappa}} t\right)\right)_{t \geq 0} \xrightarrow{\text{f.d.d.}} \big(\tilde N(t), \tilde M(t)\big)_{t \geq 0}, 
\end{align*}
 as $\kappa \rightarrow  \infty$.
\end{thm}

Recall from \eqref{eq:duality} that for each fixed $\kappa\in \N$ the process $(N^{c_{\kappa}}(t), M^{c_{\kappa}}(t))_{t \geq 0}$ is the moment dual of the seed bank diffusion $(X^{c_{\kappa}}(t), Y^{c_{\kappa}}(t))_{t \geq 0}$, where again indicate the value of the migration rate by the superscript $c_{\kappa}$. As we will see in Section \ref{sec:limitsdiff} this moment duality is the key ingredient that allows to formalize the proof of convergence of this sequence of diffusions and the existence of the limit as a  Markov process, which is not a diffusion and ``essentially'' has state space $\{0,1\}\times [0,1]$, as described in Figure \ref{fig:ancient_diff}.
\begin{thm}
\label{thm:anc_diff}
 Let $(X^{c}(t), Y^{c} (t))_{t \geq 0}$ be the seed bank diffusion given in Definition \ref{defn:system} with migration rate $c>0$. There exists a Markov process $(\tilde X(t),\tilde Y(t))_{t\geq 0}$ on $[0,1]^2$ such that for any sequence of migrations rates with $c_{\kappa}\rightarrow 0$ for $\kappa\rightarrow\infty$
 \begin{align*}
  \left(X^{c_{\kappa}}\left(\frac{1}{c_{\kappa}}t\right), Y^{c_{\kappa}} \left(\frac{1}{c_{\kappa}}t\right)\right)_{t \geq 0} \xrightarrow{f.d.d.} (\tilde X(t),\tilde Y(t))_{t\geq 0}
 \end{align*}
as $\kappa\rightarrow\infty$. Furthermore, $(\tilde X(t),\tilde Y(t))_{t\geq 0}$ is characterized as the moment dual of the ancient ancestral lines process from Definition \ref{defn:ancgen}.
\end{thm}
Much like its dual, the limit $(\bar X(t), \bar Y(t))_{t\geq 0}$ is ``degenerate'' in the sense that it does not have a generator because of the discontinuity of its semi-group in $t=0$. However, as we will prove in Proposition \ref{prop:generator_restricted}, if started in $\{0,1\}\times[0,1]$ it coincides in distribution with a jump-diffusion taking values in $\{0,1\}\times[0,1]$ whose generator is given by
\begin{align*}
  \bar A f(x,y) 	&= y(f(1,y)-f(0,y))\1_{\{0\}}(x) + (1-y)(f(0,y)-f(1,y))\1_{\{1\}}(x)\notag \\
			& \qquad + K(x-y)\frac{\partial f}{\partial y} (x,y).
\end{align*}
In particular this means that the limit $(\tilde X(t),\tilde Y(t))_{t\geq 0}$ instantaneously jumps into the smaller state space $\{0,1\}\times[0,1]$.

\begin{figure}
\center
\includegraphics[width=.8\textwidth]{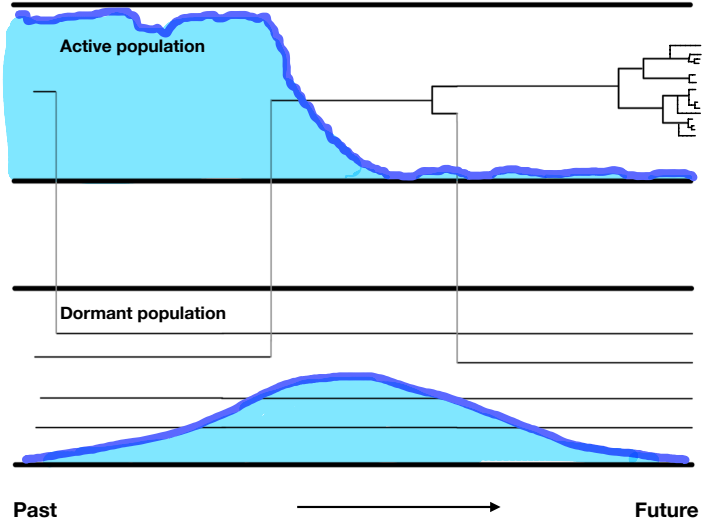}
\caption{For minuscule migration rate $c$, most of the time, the active population will be almost homogeneous, i.e.\ the frequency process of the active population will be very close to one of its boundaries. However, the (rare) migration events of the opposite type from the dormant population will prevent it from staying in that boundary. From time to time, one of these migrations might lead to a change of the predominant type in the active population. This sweep will be extremely fast (of order of the inverse of $c$), and thus instantaneous in the limit. Due to the homogeneity of the active population, the seed bank will mostly receive plants of the type dominant in the active population at that time and will thus evolve almost deterministically.} \label{fig:ancient_diff}
\end{figure}

As in the case of the time-continuous Markov chains, we again state the general result of translating the convergence through duality in \ref{subsec:general_diff} and the apply it with the \emph{ancient material scaling} in Section \ref{subsec:ancient_material_diff}.

\begin{rem}
 Since we believe the methodology used to prove Theorems \ref{thm:newcoal} and \ref{thm:anc_diff} can be applied in many situations, helping those interested in scaling limits of Markov processes that experience a separation of scales, we have separated the general methodology from the example of the \emph{ancient ancestral material scaling}.

 For time-continuous Markov chains, this is done in Section \ref{subsec:Strategy_of_proof} with its key innovation being Lemma \ref{topological}.
 
 If the processes of interest, on the other hand, are the moment duals of a sequence of time-continuous Markov chains, then the convergence in finite dimensional distributions of one sequence of processes can be translated into convergence in finite dimensional distributions for the other, propagating a separation of time-scales where applicable. The strategy of proof is to use the commutative diagram depicted in Figure \ref{fig:comdiag_convergence}. This is the content of Theorem \ref{subsec:general_diff} in Section \ref{subsec:general_diff} and allows us to prove the convergence of a family of diffusions into a non trivial, non diffusion Markov process. 
 \begin{figure}
 
\center
\includegraphics[width=.5\textwidth]{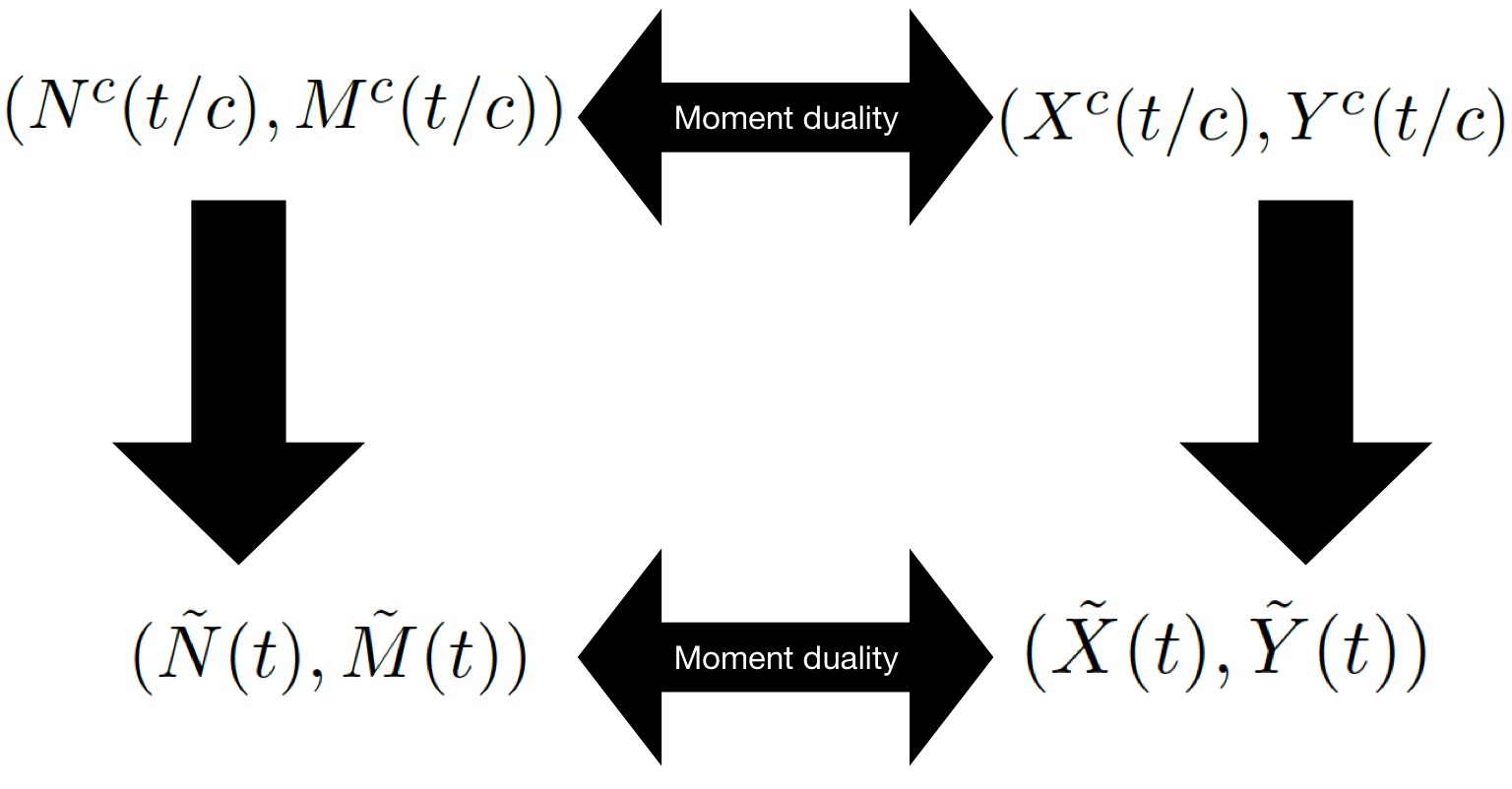}
\caption{Commutative diagram summarizing the relations between the processes considered. The moment duality of the prelimits and the limits is used to conclude the convergence in f.d.d.\ on the right from the convergence of the processes on the left.} \label{fig:comdiag_convergence}
\end{figure}
 
 Once convergence of the finite dimensional distributions is established, it is natural to wonder if it is possible to prove tightness, in order to obtain weak convergence over the Skorohod space with the $J_1$-topology. At first glance maybe surprisingly, it is not hard to see that convergence in any of the Skorohod topologies cannot hold. To see this, observe that there are \textit{sparks} occurring at the events that happen instantaneously in the limit. In these time-points the prelimiting processes visit a state outside the smaller state space of the limit process. In the coalescent set-up, for example, if there is one seed and one plant, in order to lose one block the discrete processes go from the state $(1,1)$ to $(2,0)$ and then $(1,0)$ (in quick succession). On the other hand, the limiting process goes directly from $(1,1)$ to $(1,0)$. Regardless of the time spent in the state $(2,0)$ by the prelimiting processes approaching 0, this makes convergence in any of the Skorohod topologies impossible. However, the set of such time-points has Lebesgue-measure equal to zero, whence convergence in the Meyer-Zheng topology (\cite{Kur91}, page 8) should hold. We do not analyze this in more details in this manuscript.

\end{rem}

\section{Scaling limits for time-continuous Markov chains} 
\label{sec:scaling}

As thoroughly motivated by the example of the \emph{ancient material scaling} in the introductory section, as a first step, we consider scaling limits of time-continuous Markov chains and extend the results for time-discrete Markov chains from \cite{BBE13,M98,MN16}.

\subsection{Separation of time-scales phenomena for time-continuous Markov chains - a strategy}
\label{subsec:Strategy_of_proof}

 Given a sequence of time-continuous Markov chains $(\xi^{\kappa}(t))_{t\geq 0}$, $\kappa \in \N$ with finite state-space $E$ (equipped with a metric $d$), we want to prove its convergence under some time-rescaling $(c_{\kappa})_{\kappa \in \N}$ to a limit $(\xi(t))_{t\geq 0}$ for $\kappa \rightarrow \infty$. 
 
 The idea behind these proofs has three steps:

\begin{itemize}
\item[\textbf{i)}] First, we consider {\em time discretizations} of the original time-continuous Markov chains by considering the time-discrete Markov chains $\eta^{\kappa}(i):=\xi^{\kappa}(i/a_{\kappa}), i \in \N_0$. The  non-negative sequence $(a_{\kappa})_{\kappa \in \N}$ with $a_{\kappa} \rightarrow \infty$, will be chosen to ensure the distance between the discretizations and the original processes to be sufficiently small.

\item[\textbf{ii)}] Secondly, we employ a generalization of Theorem 1 in \cite{M98}, namely Lemma 1.7 in \cite{BBE13}, to establish \emph{convergence of the discretized processes to the desired time-continuous limit} on the new time scale by speeding up the discretized processes by $b_{\kappa}=a_{\kappa}/c_{\kappa}$, that is, we establish the convergence $(\eta^{\kappa}(\lfloor b_{\kappa}t\rfloor ))_{t \ge 0} \rightarrow (\xi(t))_{t \ge 0}$ in finite dimensional distributions.

\item[\textbf{iii)}] Finally, we prove a continuity result to show that the original processes sped up by the factor $b_{\kappa}/a_{\kappa}$, i.e. $(\xi^{\kappa}(b_{\kappa}t/a_{\kappa}))_{t \ge 0}$, converges to the same limit $(\xi(t))_{t \ge 0}$ in finite dimensional distributions. 
\end{itemize}

Since the strategy is not at all restricted to the specific examples we consider, but might be of general interest, we will give the details together with the necessary results here.

{\bf Step i)} Denote by $G^{\kappa}$ the $Q$-matrix of $(\xi^\kappa(t))_{t \geq 0}$ for each $\kappa \in \N$. The rescaling sequence $(a_{\kappa})_{\kappa \in \N}$ needs to be chosen such that for $q_{\kappa}:=\max_{e \in E} \left\{-G^{\kappa}_{e,e}\right\}$
\begin{align}\label{eq:neconejump}
 a_{\kappa} \rightarrow \infty \qquad \text{and} \qquad \frac{q_{\kappa}}{a_{\kappa}} \rightarrow 0, \qquad \text{ for } \kappa \rightarrow \infty.
\end{align}

Define the time discretizations
\begin{align*}
\eta^{\kappa}(i):=\xi^{\kappa}(i/a_{\kappa}), i \in \N_0 
\end{align*}
  of the original sequence of time-continuous Markov chains $(\xi^{\kappa}(t))_{t\geq 0}$, $\kappa \in \N$. 

 As we will see in the proof of Lemma \ref{topological} in \textbf{Step iii)}, \eqref{eq:neconejump} will ensure the step-size to be sufficiently fine for the probability of a jump of $\xi^{\kappa}$ during one time-step of $\eta^{\kappa}$ to tend to 0.

{\bf Step ii)} The next step is to apply the known convergence result for time-discrete Markov chains from \cite{BBE13} to the sequence $(\eta^{\kappa})(i))_{i \in \N_0}$, the assumptions of which we summarize here for the reader's convenience.Let  $\Pi_{\kappa}$ be the transition matrix of $(\eta^\kappa(i))_{i \in \N_0}$.

First, establish a suitable decomposition of $\Pi_{\kappa}$: For the sequence $(b_{\kappa})_{\kappa \in \N}$ with $b_{\kappa}=a_{\kappa}/c_{\kappa} \rightarrow \infty$ write
\begin{equation}
\label{eq:cond1}
\Pi_\kappa = A_\kappa + \frac{B_\kappa}{b_{\kappa}} 
\end{equation}
where $A_\kappa$ is a stochastic matrix that contains only entries of order $1$ and $a_\kappa^{-1}$, and $B_\kappa$ contains only entries of order $1$ and $o(1)$. As we will see below, speeding up time by the factor $b_{\kappa}$ leads in the limit to a separation of time-scales, where the entries in $A_{\kappa}$ give rise to a projection matrix $P$ acting on the probability distributions on $E$, effectively restricting the state space of the limiting time-continuous Markov chain to a subspace of $E$, while the entries of $B_{\kappa}$ yield the infinitesimal generator. 

In order to prove this, first confirm that
\begin{equation}
\label{eq:cond2}
\lim_{C \to \infty}\lim_{\kappa \to \infty} \sup_{r \geq Ca_{\kappa}} \Vert({{A}_{\kappa}})^r -P\Vert =0
\end{equation}
for some matrix $P$. Here we equipped the matrices $A=(A(e,\bar e))_{e,\bar e \in E}$ on $E$ with the matrix norm $\Vert A\Vert:= \max_{e\in E} \sum_{\bar e\in E} |A(e,\bar e)|$.
Note that given \eqref{eq:cond2}, the matrix $P$ is necessarily a projection on $E$, i.e.\ satisfies $P^2=P$, as can be checked by a small calculation. 

Secondly, require that the matrix limit with respect to the matrix norm
\begin{equation}
\label{lemma:twoscales:eq:BNlimit}
G:=\lim_{\kappa\to\infty} P {B}_{\kappa} P \qquad \text{exists}.
\end{equation}

Since $E$ is assumed to be finite, convergence in matrix norm is equivalent to point-wise convergence. Then, by \cite[Lemma 1.7 and Remark 1.8]{BBE13}, we obtain the following convergence (with respect to the matrix norm):
\begin{equation} 
\label{eq:twoscaleslimit}
\lim_{\kappa\to\infty}\Pi_\kappa^{\lfloor t b_{\kappa}\rfloor} = \lim_{\kappa\to\infty}  \left({A}_{\kappa} + \frac{{B}_{\kappa}}{b_{\kappa}}\right)^{\lfloor t b_{\kappa}\rfloor} = Pe^{tG} =:\Pi(t) \qquad \text{for all}\;\; t > 0.
\end{equation}

Note that since $P=P^2$, we have $PG=GP=G$ and hence $Pe^{tG}=e^{tG}P = P-I+e^{tG}$ for any $t\geq 0$. In particular, $(\Pi(t))_{t\geq 0}$ with $\Pi(0):=\text{Id}_E$ is a (non-standard) semi-group and we denote by $(\xi(t))_{t\geq 0}$ the time-continuous Markov chain it generates.

If, last but not least, $\eta^{\kappa}(0) \xrightarrow{w} \xi(0)$, equation \eqref{eq:twoscaleslimit} implies 
\begin{align*}
 (\eta^\kappa(\lfloor b_{\kappa}t\rfloor))_{t\geq0} \xrightarrow{\text{f.d.d.}} (\xi(t))_{t\geq0}, \quad \text{ as } \kappa\rightarrow \infty.
\end{align*}
Here, $\xrightarrow{\text{f.d.d.}}$ denotes convergence of the processes in \emph{finite dimensional distributions}.

{\bf Step iii)} Lastly, we prove that the conditions given in \textbf{Step i)} and \textbf{ii)} are sufficient to also the convergence of original time-continuous Markov chains $(\xi^\kappa(t))_{t \ge 0}$ to the same limit  $(\xi(t))_{t \ge 0}$ on the faster time-scale $b_{\kappa}/a_{\kappa}$ as well. This is summarized in the following lemma.

\begin{lemma}
\label{topological}
Let $(\xi^\kappa(t))_{t \ge 0},\kappa\in\N$ be a sequence of time-continuous Markov chains with finite state space $E$ (equipped with some metric $d$). Let $(a_{\kappa})_{\kappa \in \N}$ and $(b_{\kappa})_{\kappa \in \N}$ be non-negative sequences such that $a_{\kappa}, b_{\kappa}/a_{\kappa} \rightarrow \infty$. 

Define the sequence of time-discrete Markov chains $(\eta^\kappa(i))_{i \in \N_0},\kappa\in\N$ by 
$$
\eta^\kappa(i):= \xi^{\kappa}\left(\frac{i}{a_{\kappa}}\right), \quad i\in \N_0.
$$
Denote by $G^{\kappa}$ the $Q$-matrix of $(\xi^\kappa(t))_{t \geq 0}$ for each $\kappa \in \N$ and set $q_{\kappa}:=\max_{e \in E} \left\{-G^{\kappa}_{e,e}\right\}$. If
\begin{itemize}
\item[a)] $\frac{q_{\kappa}}{a_{\kappa}} \rightarrow 0$ and 
\item[b)] $(\eta^\kappa(\lfloor b_{\kappa}t\rfloor))_{t \ge 0} \xrightarrow{\text{f.d.d.}} (\xi(t))_{t \ge 0}$ as $\kappa \to \infty$, and 
\end{itemize}
then also 
\begin{align*}
\left(\xi^\kappa\left(\frac{b_{\kappa}}{a_{\kappa}}t\right)\right)_{t \ge 0} \xrightarrow{\text{f.d.d.}} (\xi(t))_{t \ge 0} \quad \text{ as }\kappa \to \infty. 
\end{align*}
\end{lemma}
Hence, if the conditions given in \textbf{Step i)} and \textbf{ii)} hold, then Lemma \ref{topological} will yield the desired convergence.

\begin{proof}

Note that condition a) was chosen precisely such that 
\begin{align}\label{eq:onejump}
 \P\left\{(\xi^{\kappa}(t))_{t\geq0} \text{ has a jump in } \left(0\;,\;\frac{1}{a_{\kappa}}\right]\right\} \leq 1-\exp\left(\frac{-q_{\kappa}}{a_{\kappa}}\right)\rightarrow 0, \quad \kappa \rightarrow \infty.
 \end{align}
Observe that for the distance between $(\xi^{\kappa}(t))_{t\geq0}$ and $(\eta^{\kappa}(t))_{t\geq0}$ at any time $t \geq 0$ we have
\begin{align*}
 d\left(\xi^{\kappa}\left(\frac{b_{\kappa}t}{a_{\kappa}}\right), \eta^{\kappa}(\lfloor b_{\kappa}t\rfloor)\right)	= d\left(\xi^{\kappa}\left(\frac{b_{\kappa}t}{a_{\kappa}}\right), \xi^{\kappa}\left(\frac{\lfloor b_{\kappa}t\rfloor}{a_{\kappa}}\right)\right) > 0
\end{align*}
\emph{only if} the process $(\xi^{\kappa}(t))_{t\geq0}$ has a jump in the interval $\left(\frac{\lfloor b_{\kappa}t\rfloor}{a_{\kappa}}, \frac{b_{\kappa}t}{a_{\kappa}}\right]$. Since its length can be estimated through
\begin{align*}
 0 \leq \frac{b_{\kappa}t}{a_{\kappa}} - \frac{\lfloor b_{\kappa}t\rfloor}{a_{\kappa}} \leq \frac{1}{a_{\kappa}}
\end{align*}
we can estimate the probability of this event with \eqref{eq:onejump} and  obtain
\begin{align}\label{eq:onejump2}
 \P\left\{d\left(\xi^{\kappa}\left(\frac{b_{\kappa}t}{a_{\kappa}}\right), \eta^{\kappa}(\lfloor b_{\kappa}t\rfloor)\right)>0\right\} \leq 1-\exp\left(\frac{-q_{\kappa}}{a_{\kappa}}\right)\rightarrow 0, \quad \kappa \rightarrow \infty.
\end{align}

In order to prove the convergence of the finite dimensional distributions, recall that weak convergence of measures is equivalent to convergence in the Prohorov metric (see, e.g. \cite{W02}, Section 3.2). Hence, assumption b) yields that for all time points $0 \leq t_0, \ldots, t_l<\infty$, states $e_0, \ldots, e_l \in E$ and any $\varepsilon >0$ sufficiently small there exists a $\bar \kappa \in \N$ such that for all $\kappa \geq \bar \kappa$:
\begin{align*}
 \P\left\{\eta^{\kappa}\left(\lfloor b_{\kappa}t_0\rfloor\right)=e_0, \ldots, \eta^{\kappa}\left(\lfloor b_{\kappa}t_l\rfloor\right)=e_l\right\} \geq \P\left\{\xi^{\kappa}(t_0)=e_0, \ldots, \xi^{\kappa}(t_l)=e_l\right\} - \frac{\varepsilon}{2}.
\end{align*}
Combining this with \eqref{eq:onejump2} we see that for all time points $0 \leq t_0, \ldots, t_l<\infty$, states $e_0, \ldots, e_l \in E$ and any $\varepsilon >0$ sufficiently small there exists a $\bar \kappa \in \N$ such that for all $\kappa \geq \bar \kappa$
\begin{align*}
 \P	&\left\{\xi^{\kappa}\left(\frac{ b_{\kappa}t_0}{a_{\kappa}}\right)=e_0, \ldots, \xi^{\kappa}\left(\frac{b_{\kappa}t_l}{a_{\kappa}}\right)\right\}	\\
	& \qquad \geq \P\Bigg\{ \eta^{\kappa}(\lfloor b_{\kappa}t_0\rfloor)=e_0, \ldots, \eta^{\kappa}(\lfloor b_{\kappa}t_l\rfloor)=e_l, \\
	& \qquad \qquad \qquad d\left(\xi^{\kappa}\left(\frac{ b_{\kappa}t_0}{a_{\kappa}}\right), \eta^{\kappa}(\lfloor b_{\kappa}t_0\rfloor)\right)=\cdots = d\left(\xi^{\kappa}\left(\frac{ b_{\kappa}t_l}{a_{\kappa}}\right), \eta^{\kappa}(\lfloor b_{\kappa}t_l\rfloor)\right)= 0\Bigg\}\\
	& \qquad \geq \P\left\{ \eta^{\kappa}(\lfloor b_{\kappa}t_0\rfloor)=e_0, \ldots, \eta^{\kappa}(\lfloor b_{\kappa}t_l\rfloor)=e_l \right\} - \frac{\varepsilon}{2}\\
	& \qquad \geq \P\left\{\xi^{\kappa}(t_0)=e_0, \ldots, \xi^{\kappa}(t_l)=e_l\right\} - \varepsilon.
 \end{align*}
This implies the convergence of the finite dimensional distributions of $\left(\xi^{\kappa}\left(\frac{b_{\kappa}}{a_{\kappa}}t\right)\right)_{t \geq 0}$  to the finite dimensional distributions of $(\xi(t))_{t \geq 0}$ in the Prohorov metric and hence weakly, which completes the proof. 
\end{proof}

\subsection{The ancient ancestral lines process (and other scaling limits)}\label{subsec:ancestral_material_coal}

Let us apply these theoretical observations to the \emph{ancestral material scaling limit} discussed in Section \ref{sec:intro} to the block-counting process of the seed bank coalescent defined in Definition \ref{defn:blockcountingSBC} with vanishing migration rate $c$. If we simply let $c \to 0$, the limiting object will be a (block counting process of the) Kingman coalescent in the plant population and a constant population of seeds. However, if we speed up time by a factor $1/c \to \infty$, we obtain a new structure given in Definition \ref{defn:ancgen}, thus uncovering a separation of time-scales phenomenon. While the exchange of ancestral lineages between active and dormant states here becomes rare in the original timescale, in the new timescale, this migration will still happen at rate 1 while coalescences in the active population now occur almost instantaneously. Hence, in the limit, for each time $t > n$, there will be at most one active line. 

Theorem \ref{thm:newcoal} establishes the ancient ancestral lines process as scaling limit in finite dimensional distributions of the block-counting process of the seed bank coalescent.

\begin{proof}[Proof of \ref{thm:newcoal}]
We prove the result using the machinery outlined in the previous section with $a_{\kappa} := c^{-2}_{\kappa}$ and $b_{\kappa} := c^{-3}_{\kappa}$. W.l.o.g.\@ assume $c_{\kappa}\leq 1$, for all $\kappa \in \N$.

{\bf Step i)} In analogy to our previous notation abbreviate 
$$
(\xi^{\kappa}(t))_{t \ge 0}:=(N^{c_{\kappa}}(t), M^{c_{\kappa}}(t))_{t \geq 0}
$$ 
and consider a discretized process with time steps of length $a_{\kappa}^{-1} = c^{2}_{\kappa}$ by letting 
$$
\eta^{\kappa}(i):= \xi^{\kappa}(ic^2_{\kappa}), \qquad i \in \N_0.
$$ 
Recalling the rates of this processes as given in Definition \ref{defn:blockcountingSBC} 
\begin{align*}
 q_{\kappa}:=\max_{(n,m) \in E_{(n_0, m_0)}} \left\{-\bar A^{c_{\kappa}}_{(n,m),(n,m)}\right\} \leq \binom{n_0+m_0}{2} + c_{\kappa}(n_0+m_0) + c_{\kappa}K(n_0+m_0)
\end{align*}
whence \eqref{eq:onejump} (and therefore \eqref{eq:onejump2}) hold as required.

{\bf Step ii)}  Let $\Pi_{\kappa}$ be the transition matrix  of $(\eta^\kappa(i))_{i \in \N_0}$. 

One can calculate the transition probabilities of this chain to be
\begin{align*}
 \P	 \{\eta^\kappa(1) & = (\bar n,\bar m) \mid \eta^\kappa(0) = (n,m)\} \phantom{\Big(}	 \\
	& \qquad = \P\left\{(N^{c_{\kappa}}(c^2_{\kappa}), M^{c_{\kappa}}(c^2_{\kappa})) = (\bar n,\bar m) \mid (N^{c_{\kappa}}(0), M^{c_{\kappa}}(0)) = (n,m)\right\} \phantom{\Big(}\\
	& \qquad = \begin{cases}
			  \binom{n}{2}c^2_{\kappa} + o(c^3_{\kappa}), 	& \text{ if } \bar n = n - 1,\, \bar m = m,\\
			  c_{\kappa}nc^2_{\kappa} + o(c^3_{\kappa}),	& \text{ if } \bar n = n - 1,\, \bar m = m+1,\\
			  c_{\kappa}Kmc^2_{\kappa} + o(c^3_{\kappa}),	& \text{ if } \bar n = n+1,\, \bar m = m-1,\\
			  1 -  \binom{n}{2}c^2_{\kappa} - c_{\kappa}nc^2_{\kappa} - c_{\kappa}Kmc^2_{\kappa} + o(c^3_{\kappa}),	& \text{ if } \bar n = n,\, \bar m = m,\\
			  0,				& \text{ otherwise.}
            \end{cases}
\end{align*}
for any sensible $(n,m), \, (\bar n, \bar m) \in E_{(n_0,m_0)}$, recalling the convention of $\binom{n}{2}=0$ for $n\leq 1$. Therefore, we obtain the decomposition
\begin{equation*}
\Pi_{\kappa} = A_{\kappa} + \frac{B_\kappa}{b_{\kappa}} 
\end{equation*}
with $b_{\kappa} =  c^{-3}_{\kappa}$ as defined above and
\begin{equation*}\label{eq:Arates}
{(A_{\kappa})}_{(n,m),(\bar n, \bar m)} = \begin{cases}
                                           \binom{n}{2}c^2_{\kappa},		& \text{ if } \bar n = n - 1,\, \bar m = m,\\
                                           1 -  \binom{n}{2}c^2_{\kappa},	& \text{ if } \bar n = n,\, \bar m = m,\\
                                           0,				& \text{ otherwise,}
                                          \end{cases}\\
\end{equation*}
and 
\begin{equation}\label{eq:Brates}
{(B_{\kappa})}_{(n,m),(\bar n, \bar m)} = \begin{cases}
                                           n + o(1),	& \text{ if } \bar n = n - 1,\, \bar m = m+1,\\
                                           Km + o(1),	& \text{ if } \bar n = n+1,\, \bar m = m-1,\\
                                           - n - Km + o(1),	& \text{ if } \bar n = n,\, \bar m = m,\\
                                           o(1),	& \text{ otherwise.}
                                          \end{cases}
\end{equation}
In order to apply the convergence result, we first need to check condition \eqref{eq:cond2} for our set-up, which reads
\begin{equation}
\label{eq:cond2_spec}
\lim_{C \to \infty}\lim_{\kappa \to \infty} \sup_{r \geq Cc^{-2}_{\kappa}} \Vert({A}_{\kappa})^r - P\Vert = 0
\end{equation}
for $P$ given in  \eqref{eq:Pancient}. Note that $A_{\kappa}$ is a stochastic matrix and denote by $(Z^{\kappa}_r)_{r \in \N_0}$ the Markov chain associated to it.
This is a pure death process in the first component and constant in the second. 
Then, by the definition of matrix norm, we get
\begin{align*}
\Vert (A_{\kappa})^r -P \Vert 	& = \max_{(n,m) \in E_{(n_0, m_0)}} \sum_{(\bar n, \bar m) \in E_{(n_0, m_0)}}  |(A_{\kappa})^r_{(n,m), (\bar n, \bar m)} - P_{(n,m), (\bar n, \bar m)}| \\
				& = \max_{n \geq 1, m \geq 0} \Big( \vert(A_{\kappa})^r_{(n,m), (1,m)} - 1 \vert + \sum_{\bar n = 2}^{n} \vert (A_{\kappa})^r_{(n,m), (\bar n,m)} - 0 \vert\Big) \\
				& = \max_{n \geq 1, m \geq 0} 2\Big( 1-(A_{\kappa})^r_{(n,m), (1,m)}\Big) \\
				& = 2 \max_{n \geq 1, m \geq 0} \P \big\{ Z^{\kappa}_r \neq (1,m) \mid Z^{\kappa}_0 = (n,m)\big\}.
\end{align*}

Observe that for all $n \geq 2$ (and all $m \geq 0$)
\begin{align*}
{A_{\kappa}}_{(n,m), (n-1,m)} = \binom{n}{2}c^{2}_{\kappa} \geq c^{2}_{\kappa}.
\end{align*} 
Hence, since we start from $n_0$ active individuals, the number of time-steps required until full coalescence is dominated by the sum of $n_0-1$ independent geometric random variables $\gamma_1^{\kappa},\dots ,\gamma_{n_0-1}^{\kappa}$ with success probability $ c^{2}_{\kappa}$. By Markov's inequality, we get 
\begin{align*}
\sup_{\kappa \in \N} \P \left\{ \gamma_1^{\kappa} + \dots + \gamma_{n_0-1}^{\kappa} \geq Cc^{-2}_{\kappa} \right\} 	& \leq \frac{c^{2}_{\kappa}}{C} \E \big[\gamma_1^{\kappa} + \dots + \gamma_{n_0-1}^{\kappa} \big]
															 = \frac{(n_0-1)c^{2}_{\kappa}}{Cc^{2}_{\kappa}} = \frac{(n_0-1)}{C} 
\end{align*}
and with it
\begin{align*}
\lim_{C \to \infty}\lim_{\kappa \to \infty} \sup_{r \geq Cc^{-2}_{\kappa}} \Vert({A}_{\kappa})^r - P\Vert \leq \lim_{C \to \infty}\lim_{\kappa \to \infty} \sup_{r \geq Cc^{-2}_{\kappa}} \frac{(n_0-1)}{C}  = 0
\end{align*}
which gives \eqref{eq:cond2_spec}. We are now left to establish the matrix-norm limit \eqref{lemma:twoscales:eq:BNlimit}, that is, show that
\begin{equation}
\lim_{N\to\infty} P {B}_{N} P \qquad \text{exists}
\end{equation}
and coincides with the $G$ in Definition \ref{defn:ancgen}. For this, notice that $B_{\kappa}$ converges for $\kappa \to \infty$ uniformly and in the matrix norm (recalling that the state space $E_{(n_0,m_0)}$ is finite), and define 
\begin{align*}
 B \; := \;\lim_{\kappa \rightarrow \infty} B_{\kappa} \; = \; \begin{cases}
                                                     n,		& \text{ if } \bar n = n-1,\, \bar m = m+1,\\
                                                     Km,	& \text{ if } \bar n = n+1,\, \bar m = m-1,\\
                                                     -n-Km,	& \text{ if } \bar n = n,\, \bar m = m,\\
                                                     0,		& \text{ otherwise.}
                                                    \end{cases}
\end{align*}
Through careful calculations one confirms 
\begin{align}\label{eq:PBP_anc}
G=PBP 
\end{align}
 and thus \eqref{lemma:twoscales:eq:BNlimit}. As described in the previous section, \cite[Lemma 1.7 and Remark 1.8]{BBE13} then yields
\begin{align*} 
\lim_{\kappa\to\infty}\Pi_\kappa^{\lfloor t c^{-3}_{\kappa}\rfloor} = \lim_{\kappa\to\infty}  \left({A}_{\kappa} + c^{3}_{\kappa}{B}_{\kappa}\right)^{\lfloor t c^{-3}_{\kappa}\rfloor} = Pe^{tG} =:\Pi(t) \qquad \text{for all}\;\; t > 0,
\end{align*}
which given $\eta^{\kappa}(0) =(N^{c_{\kappa}}(0), M^{c_{\kappa}}(0)) = (\tilde N(0), \tilde M(0))$ then implies 
\begin{align*}
 (\eta^\kappa(\lfloor c^{-3}_{\kappa}t\rfloor))_{t\geq0} \rightarrow (\tilde N(t),\tilde M(t))_{t\geq0} \quad \text{in finite dimensional distributions, as } \kappa\rightarrow \infty,
\end{align*}
where $(\tilde N(t),\tilde M(t))_{t\geq0} $ is the ancient ancestral lines process defined in Definition \ref{defn:ancgen}.

{\bf Step iii)} Since we have proven the necessary assumptions in Step i) and ii), Lemma \ref{topological} implies
\begin{align*}
\left(N^{c_{\kappa}}(c^{-1}_{\kappa}t), M^{c_{\kappa}}(c^{-1}_{\kappa}t)\right)_{t \geq 0} 	& = \left(\xi^{c_{\kappa}}\left(\frac{c^{-3}_{\kappa}}{c^{-2}_{\kappa}}t\right)\right)_{t \geq 0} \longrightarrow \left(\tilde N(t), \tilde M(t)\right)_{t \geq 0} 
\end{align*}
in finite dimensional distributions for $\kappa \rightarrow \infty$ and the proof is complete.
\end{proof}

\begin{rem}[Imbalanced Island Size]\label{rem:imbalanced_island_size_coal}
It is straightforward to pursue the same considertation for the \emph{two-island model} and its \emph{structured coalescent} \cite{H94, Not90}. The \emph{two-island model} considers two poulations much like the \emph{seed bank model}, but allows for coalescence in the second population. Its genalogy is the given by the the structured coalescent, whose blockcounting process allows for the same transition rates described in \eqref{eq:ratesSBC} adding $r_{(n,m),(n'm')} = \binom{m}{2}$ for $n'=n$ and $m'=m-1$, i.e.\ coalescence in the second island.

Scaling the migration rate $c\rightarrow 0$ while speeding up time by $1/c\rightarrow \infty$ as we previsouly did for the block counting process of the seed bank coalescent will lead to a structure with instantaneous coalescences in \emph{both} islands, leaving us with a single line migrating between them.

Much more interesting in this set-up ist to consider a two-island model with \emph{different} scalings of the coalescence rates in the islands. The parameters $\alpha$ resp. $\alpha'$ are associated with the notion of \emph{effective population size} (cf.\ \cite{W08}) so a different scaling corresponds to a significant difference in population size on the two islands. If, in addition to $c\rightarrow 0$ we assume the coalescence  rate $\alpha'>0$ in the second island to scale as $c$, i.e. $\alpha'/c \rightarrow 1$, the result is a two-island model with instantaneous coalescences in the first island, but otherwise `normal' migration and coalescence behavior in the second. For more precision, denote by $(N^{c,\alpha'}(t), M^{c,\alpha'}(t))_{t \geq 0}$ the block-counting process of the structured coalescent as defined above with migration rate $c>0$ and coalescence rate $\alpha'>0$ in the second island and assume that it starts in  some $(n_0, m_0) \in \N\times\N$, $\P$-a.s.. 

Define $(\hat N(t), \hat M(t))_{t \geq 0}$ to be the  the time-continuous Markov chain with initial value $(\hat N(0), \hat M(0))=(n_0, m_0)$, taking values in the state space $ E_{(n_0,m_0)}:=\lbrace 0, \dots, n_0+m_0 \rbrace^2$, with transition matrix $ \Pi(t):=Pe^{tG}$, for $t>0$ and $\Pi(0)$ equal to the identity on $E$, where $P$ is given by \eqref{eq:Pancient} as before and  $G$ is a matrix of the form 
\begin{align*}
 G_{(n,m),(\bar n,\bar m)} := \begin{cases}
                              Km + \binom{m}{2},	& \text{ if } \bar n = 1,\, n \geq 1,\, \bar m = m-1,\\
                              Km,	& \text{ if } \bar n = 1,\, n = 0,\, \bar m = m-1,\\
                              \binom{m}{2},	& \text{ if } \bar n = 0,\, n = 0,\, \bar m = m-1,\\
                              n,	& \text{ if } \bar n = 0,\, n \geq 1,\, \bar m = m+1,\\
                              -\binom{m}{2} - n - Km,	& \text{ if } \bar n = 1,\, n \geq 1,\, \bar m = m,\\
                              -\binom{m}{2} - Km,	& \text{ if } \bar n = n = 0,\, \bar m = m,\\
                              0,	& \text{ otherwise.}
                             \end{cases}
\end{align*}
Then, for any sequence of migration rates $(c_{\kappa})_{\kappa \in \N}$ and any sequence of coalescence rates $(\alpha'_{\kappa})_{\kappa \in \N}$ with $c_{\kappa} \rightarrow 0$ and $c_{\kappa}/\alpha'_{\kappa} \rightarrow 1$  for $\kappa \rightarrow \infty$ 
\begin{align*}
\left(N^{c_{\kappa}, \alpha'_{\kappa}}\left(\frac{1}{c_{\kappa}} t\right), M^{c_{\kappa}, \alpha'_{\kappa}}\left(\frac{1}{c_{\kappa}} t\right)\right)_{t \geq 0} \xrightarrow{\text{f.d.d.}} \big(\hat N(t), \hat M(t)\big)_{t \geq 0}, 
\end{align*}
 in finite dimensional distributions as $\kappa \rightarrow  \infty$.
The proof is analogous to that of Theorem \ref{thm:newcoal}. Considering, again, the sequences $a_{\kappa} := c^{-2}_{\kappa}$ and $b_{\kappa} := c^{-3}_{\kappa}$, $A_{\kappa}$ and $P$ coincide with those in the proof of Theorem \ref{thm:newcoal}, the hardest work has already been done and we ommitt the details of the proof here.  
\end{rem}

\section{Scaling limits for the diffusion}\label{sec:limitsdiff}

We would now also like to observe similar scaling limits for the diffusion \eqref{eq:system}. As we saw in the case of genealogies, rescaling time may lead to a limiting process that is still Markovian, but whose semi-group is not standard. We can, however, use \emph{moment-duality} to obtain this limit.

\subsection{Convergence in finite dimensional distributions from duality}\label{subsec:general_diff}
 We here present a general result on how to obtain convergence in finite dimensional distributions from moment duality and the analogous convergence of the dual process. This result is independent of whether time is rescaled, too, or not. It is, however, of particular interest in that case, since it might lead to limiting objects, that are rather ``ill-behaved'' and we will see examples in Section \ref{subsec:ancient_material_diff} where the limit does not have a generator, hence more standard ways of proving convergence through generator convergence fail.

For any vectors $ n:=(n_1, \ldots, n_d) \in \N_0^d$ and $ x:=(x_1, \ldots, x_d) \in [0,1]^d$, define the \emph{mixed-moment function $\mathfrak m$} as $\mathfrak m( x,  n) := x_1^{n_1}\cdots x_d^{n_d}$. 
\begin{thm}\label{thm:fdd_diff}
 Let $(\zeta_{\kappa}(t))_{t\geq 0}$, $\kappa \in \N_0$, be a sequence of Feller Markov processes taking values in $[0,1]^d$ (for some $d \in \N$) and $(\xi^{\kappa}(t))_{t\geq 0}$, $\kappa \in \N_0$, a sequence of Markov chains with values in $\N_0^d$ such that they are pairwise moment duals, i.e.
 \begin{align*}
  \forall \kappa \in \N_0\;, \forall t\geq 0\; \forall x \in [0,1]^d, n \in \N_0^d:\; \E_n[\mathfrak m(x,\xi^{\kappa}(t))]]=\E^x[\mathfrak m(\zeta_{\kappa}(t),n)].
 \end{align*}
As usual, $\P_n$ and $\P^x$ denote the distributions for which $\xi$, resp. $\zeta$, start in $n$, resp. $x$.

If $(\xi^{\kappa})_{\kappa \in \N_0}$ converges to some Markov chain $\xi$ \emph{in finite dimensional distributions}, then there exists a Markov process $\zeta$ with values in $[0,1]^d$ such that it is the limit in finite dimensional distributions of $(\zeta_{\kappa})_{\kappa \in \N_0}$ and the moment dual to $\xi$, i.e.
\begin{align}\label{eq:duality_diff_lim}
 \forall t\geq 0\; \forall x \in [0,1]^d, n \in \N_0^d:\; \E_n[\mathfrak m(x,\xi(t))]]=\E^x[\mathfrak m(\zeta(t),n)].
\end{align}

\end{thm}

\begin{rem}
 At a first glance one might suspect that this result should also hold in a more general set-up as long as the duality function used yields convergence determining families for the respective semi-groups. Indeed, most of the steps of the proof would still go through. However, note that we did not assume existence of a Markovian limit beforehand. For this we use the solvability of Hausdorff's moment problem on $[0,1]^d$ \cite{MultDimMomProb}, which is a match to the moment duality function in our theorem.
\end{rem}

\begin{proof}
 The proof can roughly be split into three steps: We first use duality to prove the convergence of the \emph{one-dimensional} distributions of $(\zeta_{\kappa})_{\kappa\in \N_0}$. This, together with the Markov property will give us the convergence of the \emph{finite dimensional distributions} of $(\zeta_{\kappa})_{\kappa \in \N_0}$ to a family of limiting distributions. Then we prove \emph{consistency} of the respective limiting measures and hence by Kolmogorov's Extension-Theorem the existence of a limiting \emph{process} $\zeta$, which must then be Markovian.
 
 Since the mixed-moment function $\mathfrak m$ is continuous and bounded as a function on $\N_0^d$, the convergence of the finite dimensional distributions of $(\xi^{\kappa})_{\kappa \in \N}$ and the assumed moment duality yield
 \begin{align}\label{eq:dualityconv1}
  \E^x[\mathfrak m(\zeta_{\kappa}(t),n)] = \E_n[\mathfrak m(x,\xi_{\kappa}(t))] \xrightarrow{\kappa \rightarrow \infty} \E_n[\mathfrak m(x,\xi(t))] :=\gamma(n,x,t)
 \end{align}
 for any $t \geq 0$, $x \in [0,1]^d$ and $n \in \N_0^d$. The unique solvability of the Hausdorff moment problem on $[0,1]^d$ \cite{MultDimMomProb} then gives the existence of a distribution  $\mu^{x,t}$ on $([0,1], \mathfrak B([0,1]))$ (where $\mathfrak B$ is the Borel-$\sigma$-algebra) such that $\gamma(n,x,t) = \int_{[0,1]^d} \mathfrak m(\bar x,n)\dd \mu^{x,t}(\bar x)$. Since the polynomials are dense in the continuous functions, \eqref{eq:dualityconv1} implies the convergence of the one-dimensional distributions to $(\mu^{x,t})_{t\geq 0}$ (for each starting point $x \in [0,1]^d$).

 To check the convergence in finite dimensional distributions, let $P_{\kappa}$ be the probability transition function of $\zeta_{\kappa}$ and recall that we assumed them to be Feller. For $0 \leq t_1 < \ldots < t_l <\infty$, $x \in [0,1]^d$ and $n_1, \ldots, n_l \in [0,1]^d$ then observe
 \begin{align}\label{eq:convfdd}
  \E^x[&\mathfrak m(\zeta_{\kappa}(t_1),n_1)\cdots\mathfrak m(\zeta_{\kappa}(t_l),n_l)] \notag\\
  & = \int_{[0,1]}\int_{[0,1]}\cdots\int_{[0,1]}\mathfrak m(\bar x_1, n_1) \cdots \mathfrak m(\bar x_l,n_l) P_{\kappa}(\bar x_{l-1}, t_l-t_{l-1}, \dd \bar x_l)\cdots P_{\kappa}(x, t_1, \dd \bar x_1)\notag\\
  & \xrightarrow{ \kappa \rightarrow \infty}: \gamma(n_1, \ldots, n_l,x,t_1, \ldots, t_l). 
 \end{align}
 The convergence (to some constant $\gamma(n_1, \ldots, n_l,x,t_1, \ldots, t_l)$) follows from the convergence of the one-dimensional distributions if one observes using the Lebesgue convergence theorem that weak convergence also implies the convergence of integrals when the integrand itself converges to a continuous and bounded function. Since our processes are Feller, we can iterate this argument and obtain the convergence above.
 
 Again, by the unique solvability of the Hausdorff moment problem \cite{MultDimMomProb} we thus obtain the existence of a measure $\mu^{I,x}$ on $([0,1]^{I}, \mathfrak B([0,1])^{\otimes I})$for any finite set of indices $I=\{t_1, \ldots, t_l\} \subset [0,\infty)$ and starting point $x\in [0,1]^d$ and \eqref{eq:convfdd} implies the convergence of the finite-dimensional distributions of $(\zeta_{\kappa})_{\kappa \in N}$ to a respective $\mu^{I,x}$. Since these $\mu^{I,x}$ are the limits of a consistent family they are themselves consistent and with Kolmogorov's Extension-Theorem there exists a unique measure $\mu^x$ on the product-space $([0,1]^{[0,\infty)}, \mathfrak B([0,1])^{\otimes [0,\infty)})$ which is the distribution of the desired process $\zeta$. Its Markovianity follows from the respective property of the $(\zeta_{\kappa})_{\kappa \in \N_0}$
 
 The duality of $\zeta$ and $\xi$ follows from the duality of the prelimiting processes.  
\end{proof}

\subsection{Ancient ancestral material}\label{subsec:ancient_material_diff}

As an application of Theorem \ref{thm:fdd_diff} we considered the diffusion \eqref{eq:system} with the scaling regime of Section \ref{subsec:ancestral_material_coal}, namely, when the migration rate $c \to 0$ while simultaneously speeding up time by a factor $1/c \to \infty$ and obtained Theorem \ref{thm:anc_diff} in Section \ref{sec:intro} stating the convergence of the rescaled diffusions to a Markovian limit $(\tilde X(t), \tilde Y(t))_{t \geq 0}$.

\begin{proof}[Proof of Theorem \ref{thm:anc_diff}]
 Since the moment duality of the block-counting process of the seed bank coalescent and the seed bank diffusion \cite[Thm.\ 2.8]{BGCKW16} holds for every time $t\geq 0$, it is preserved for the time-changed processes $\left(N^{c_{\kappa}}\left(t/c_{\kappa}\right), M^{c_{\kappa}} \left(t/c_{\kappa}\right)\right)_{t \geq 0}$ and $\left(X^{c_{\kappa}}\left(t/c_{\kappa}\right), Y^{c_{\kappa}} \left(t/c_{\kappa}\right)\right)_{t \geq 0}$. Together with Theorem \ref{thm:newcoal} all assumptions of Theorem \ref{thm:fdd_diff} hold and we get the existence of a Markov process $(\tilde X(t),\tilde Y(t))_{t\geq 0}$ that is the dual of $(\tilde N(t),\tilde M(t))_{t\geq 0}$. The uniqueness of the moment dual of a Markov process proves that the limit does not depend on the choice of scaling sequence $(c_{\kappa})_{\kappa \in \N_0}$. 
\end{proof}

Moment duality now allows us to translate our knowledge about the ancient ancestral line process $(\tilde N(t), \tilde M(t))_{t \geq 0}$ (Definition \ref{defn:ancgen}) to $(\tilde X(t), \tilde Y(t))_{t \geq 0}$. More precisely, since \eqref{eq:duality_diff_lim} holds in particular for $t >0 $, $m=0$ and any $n\geq 1$, $x,y \in [0,1]$ we see
 \begin{align*}
 \E^{x,y}[\tilde X(t)^n] 	& =  \E_{n,0}[x^{\tilde N(t)}y^{\tilde M(t)}] \\
				& = x\P_{n,0}(\tilde N(t) = 1, \tilde M(t) = 0) + y\P_{n,0}(\tilde N(t) = 0, \tilde M(t) = 1) \\
				& = x (Pe^{tG})_{(n,0),(1,0)}+y (Pe^{tG})_{(n,0),(0,1)} = x (e^{tG})_{(1,0),(1,0)} + y(e^{tG})_{(1,0),(0,1)}.
 \end{align*}
We used the fact, that the first component of the ancient ancestral lines process immediately takes values in $\{0,1\}$ in the second equality and the definition of the projection in the last equality. Since the right-hand-side does not depend on $n \geq 1$, we can conclude that 
\begin{align}\label{eq:obs_anc_X}
\tilde X(t) \in \{0,1\}\quad \P^{x,y}\text{-a.s.\ for any } t>0 \text{ and any }(x,y)\in [0,1]^2. 
\end{align}

This small observation has an important consequence: Much like in the case of its dual $(\tilde N(t), \tilde M(t))_{t \geq 0}$, the transition function of the ancient ancestral material process $(\tilde X(t), \tilde Y(t))_{t \geq 0}$ is not right-continuous in 0 and therefore $(\tilde X(t), \tilde Y(t))_{t \geq 0}$ does not have a classical generator.
Intuitively the reproduction mechanism (in the active population) acts so fast, that fixation (or extinction) \emph{in the active population} happens instantaneously. Whenever there is an invasion from the seed bank of the type extinct in the active population, its chances of fixating are proportional to the fraction of the type in the \emph{dormant} population. The limit is therefore a pure jump process in the active component that moves between the states $0$ and $1$, while the seed bank component retains its classical behavior. We can formalize this observation if we restrict the process to the smaller state space $\{0,1\}\times[0,1]$, see Proposition \ref{prop:generator_restricted} below.

\begin{defn}\label{defn:anc_restr}
 Let $(\bar N(t), \bar M(t))_{t \geq 0}$ be the Markov chain on $\{0,1\}\times\N_0$ given by the (conservative) $Q$-matrix
\begin{align*}
 \bar G_{(n,m),(\bar n,\bar m)} = \begin{cases}
                               Km,	& \text{ if } \bar n = 1,\, n \in \{0,1\},\, \bar m = m-1,\\
                              1,	& \text{ if } \bar n = 0,\, n = 1,\, \bar m = m+1,\\
                              -n-Km,	& \text{ if } \bar n = n ,\, \bar m = m,\\
                              0,	& \text{ otherwise.}
                             \end{cases}
\end{align*}
for any $(n,m),(\bar n,\bar m) \in \{0,1\}\times\N_0$. 

On the other hand, let $(\bar X(t), \bar Y(t))_{t \geq 0}$ be the Markov process on $\{0,1\}\times [0,1]$ with generator
\begin{align*}
 \bar A f(x,y) 	&= y(f(1,y)-f(0,y))\1_{\{0\}}(x) + (1-y)(f(0,y)-f(1,y))\1_{\{1\}}(x)\notag \\
			& \qquad + K(x-y)\frac{\partial f}{\partial y} (x,y).
\end{align*}
\end{defn}

Note that $\bar G$ given above is the restriction of $G$, the ``generator'' of $(\tilde N(t),\tilde M(t))_{t\geq 0}$ from Definition \ref{defn:ancgen}. Indeed, these processes are essentially the ancestral material processes when started in the smaller state-space:

\begin{prop}\label{prop:generator_restricted}
The processes $(\bar N(t), \bar M(t))_{t \geq 0}$ and $(\bar X(t), \bar Y(t))_{t \geq 0}$ from Definition \ref{defn:anc_restr} are moment duals, i.e.\ 
\begin{align}\label{eq:duality_anc_restr}
 \forall t\geq 0\; \forall (x,y) \in [0,1]^2, (n,m) \in \N_0^2:\; \E_{n,m}\left[x^{\bar N(t)}y^{\bar M(t)}\right]=\E^{x,y}\left[n^{\bar X(t)}m^{\bar M(t)}\right].
\end{align}
$(\bar N(t), \bar M(t))_{t \geq 0}$ coincides in distribution with $(\tilde N(t), \tilde M(t))_{t \geq 0}$ if (both are) started in the reduced state-space $\{0,1\}\times\N_0$. 

Likewise, $(\bar X(t), \bar Y(t))_{t \geq 0}$ coincides in distribution with $(\tilde X(t), \tilde Y(t))_{t \geq 0}$ if (both are) started in the reduced state-space $\{0,1\}\times[0,1]$. 
\end{prop}
Moment duality of the involved processes will be important for the proof of the last statement.

\begin{figure}\label{fig:mutcoalescent}

\center

\includegraphics[width=.6\textwidth]{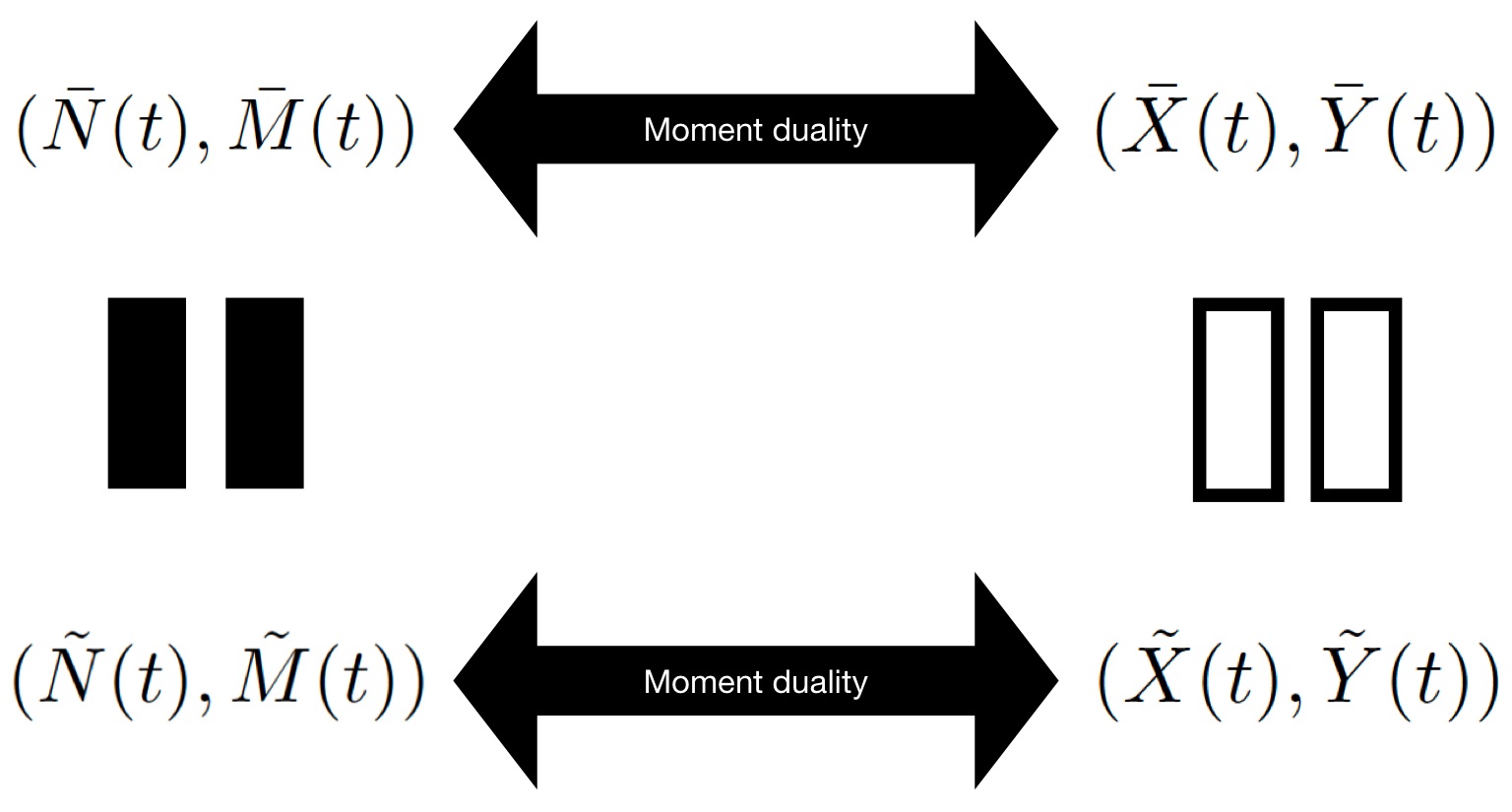}

\caption{Strategy of the proof of Proposition \ref{prop:generator_restricted}. The moment duality of $(\tilde N(t), \tilde M(t))_{t \geq 0}$ and $(\tilde X(t), \tilde Y(t))_{t \geq 0}$ is consequence of Theorem \ref{thm:anc_diff}. The semigroups of $(\tilde N(t), \tilde M(t))_{t \geq 0}$ and $(\bar N(t), \bar M(t))_{t \geq 0}$ when started in the reduced state-space $\{0,1\}\times\N_0$. We prove the moment duality of $(\bar N(t), \bar M(t))_{t \geq 0}$ and $(\bar X(t), \bar Y(t))_{t \geq 0}$, which allows us to conclude, that the semigroups of $(\tilde X(t), \tilde Y(t))_{t \geq 0}$ and $(\bar X(t), \bar Y(t))_{t \geq 0}$ also agree when started in $\{0,1\}\times[0,1]$.}

\end{figure}

\begin{proof}
The duality of $(\bar N(t),\bar M(t))_{t\geq 0}$ and $(\bar X(t),\bar Y(t))_{t\geq 0}$ can be proven through the standard method of generator calculations: Applying $\bar A$ to $S((x,y),(n,m)):=x^ny^m$ for $(n,m) \in \{0,1\}\times\N_0$ and $(x,y) \in \{0,1\}\times [0,1]$, as a function in $(x,y)$ yields
\begin{align*}
 \bar A S ((x,y),(n,m)) 	& = y(y^m-0^ny^m)\1_{\{0\}}(x) + (1-y)(0^ny^m-y^m)\1_{\{1\}}(x)\\
				& \qquad + K(x-y)x^nmy^{m-1}\\
				& = -(x-y)(y^m-0^ny^m)\1_{\{0\}}(x) + (x-y)(0^ny^m-y^m)\1_{\{1\}}(x)\\ 
				& \qquad + K(x-y)x^nmy^{m-1}\\
				& = -n(x-y)y^m + Km(x-y)x^ny^{m-1}\\
				& = Kmx^{n+1}y^{m-1} + (-Kmx^n - nx)y^m + ny^{m+1}
\end{align*}
where we continue to use $0^0=1$, the fact that $n \in \{0,1\}$ and simply sorted the terms by powers of $y$ for easier comparison in the last line.

On the other hand, if we apply $\bar G$ to $S$ as a function in $(n,m)\in\{0,1\}\times N_0$, we get
\begin{align*}
 \bar G S ((x,y),(n,m))	& = Km(xy^{m-1}-x^ny^m) + 1(y^{m+1}-xy^m)\1_{\{1\}}(n) \\
			& = Km(xy^{m-1}-xy^m)\underbrace{\1_{\{1\}}(n)}_{=n} + Km(xy^{m-1}-y^m)\underbrace{\1_{\{0\}}(n)}_{=1-n} \\
			& \qquad + 1(y^{m+1}-xy^m)\underbrace{\1_{\{1\}}(n)}_{=n}\\
			& = Kmxy^{m-1} + (-Kmnx -Km(1-n)-nx)y^m + ny^{m+1}.
\end{align*}
A close look noting that for our choices of variables we have $x^{n+1}=x$ and $nx+(1-n) = x^n$ shows that the two coincide. Since the duality function is bounded and continuous we conclude the proof with 

\cite[Prop.\ 1.2]{JK14}. 

Let $\bar f:\{0,1\}\times\N_0 \rightarrow \R$ and define $f$ as $f(n,m):=\bar f(n,m)$ for $(n,m) \in \{0,1\}\times \N_0$ and 0 otherwise. Recall that $Pe^{tG}$ is the semi-group of $(\tilde N(t),\tilde M(t))_{t\geq 0}$ from Definition \ref{defn:ancgen}. Since $G_{(n,m),(\bar n, \bar m)}=\bar G_{(n,m),(\bar n, \bar m)}$ for any $(n,m) \in \{0,1\}\times \N_0$, we have
\begin{align*}
 G f (n,m) = \bar G \bar f (n,m), \qquad (n,m) \in \{0,1\}\times \N_0.
\end{align*}
As we also know that $G = PBP$ from \eqref{eq:PBP_anc}, it follows that 
\begin{align*}
 G f (n,m) = PBP f (n,m) \in \{0,1\}\times \N_0
\end{align*}
whence the semi-groups of $(\tilde N(t),\tilde M(t))_{t\geq 0}$ and $(\bar N(t),\bar M(t))_{t\geq 0}$ coincide on such $f$ and the two processes are equal in distribution when started in $(n,m) \in \{0,1\}\times \N_0$, as claimed.

This also implies 
\begin{align} \label{eq:anc_diff_equals_restr}
  \E^{x,y}\left[n^{\tilde X(t)}m^{\tilde M(t)}\right]=\E_{n,m}\left[x^{\tilde N(t)}y^{\tilde M(t)}\right]=\E_{n,m}\left[x^{\bar N(t)}y^{\bar M(t)}\right]=\E^{x,y}\left[n^{\bar X(t)}m^{\bar M(t)}\right]
\end{align}
for all $t\geq 0$ and all $(x,y) \in \{0,1\}\times[0,1]$ and $(n,m) \in \{0,1\}\times\N_0$, where we used the dualities from Theorem \ref{thm:anc_diff} and Proposition \ref{prop:generator_restricted} in the first, respectively last equality. 

Recall from \eqref{eq:obs_anc_X}, that for any $t>0$ we have $(\tilde X(t),\tilde Y(t)) \in \{0,1\}\times[0,1]$, $\P^{x,y}$-a.s., $(x,y) \in [0,1]^2$. Since a distribution on $\{0,1\}\times[0,1]$ is uniquely determined by its moments of order $(n,m)\in \{0,1\}\times \N_0$, \eqref{eq:anc_diff_equals_restr} implies that $(\tilde X(t),\tilde Y(t)) \sim (\bar X(t),\bar Y(t))$ for any $t>0$ (when started in the same $(x,y) \in \{0,1\}\times[0,1]$). Since they are both Markovian, this implies that the distributions of $(\bar X(t), \bar Y(t))_{t \geq 0}$ and $(\tilde X(t), \tilde Y(t))_{t \geq 0}$ coincide when started in the reduced state-space $\{0,1\}\times[0,1]$. 
\end{proof}

\begin{rem}[Imbalanced Island size - Part 2]\label{rem:imbalanced_island_size_diff}
 We return to the example discussed in Remark \ref{rem:imbalanced_island_size_coal} of the two-island model and its close relation to the seed bank model.
 The frequency process of allele $a$ is then described by the \emph{two-island diffusion} \cite{KZH08}
\begin{align}
\label{eq:twoisland_diff}
\text{d} X(t) & = (Y(t) -X(t))\text{d}t + \alpha\sqrt{X(t)(1-X(t))}\text{d}B_t, \notag \\[.1cm]
\text{d} Y(t) & = K(X(t) -Y(t))\text{d}t + \alpha'\sqrt{X(t)(1-X(t))}\text{d}B'_t,
\end{align}
where $B$ and $B'$ are independent Brownian Motions.

Again, the interesting consideration here is to use \emph{different} scalings of the coalescence rates in the islands, i.e.\ different scalings for $\alpha$ and $\alpha'$. The parameters $\alpha$ resp. $\alpha'$ are associated with the notion of \emph{effective population size} (cf.\ \cite{W08}) so a different scaling corresponds to a significant difference in population size on the two islands. If, in addition to $c\rightarrow 0$ we assume the coalescence  rate $\alpha'>0$ in the second island to scale as $c$, i.e. $\alpha'/c \rightarrow 1$, the result is a two-island model with jumps in the first island, but otherwise `normal' migration and diffusive behavior in the second. For more precision, denote by $(X^{c,\alpha'}(t), Y^{c,\alpha'}(t))_{t \geq 0}$ the \emph{two-island diffusion} with migration rate $c>0$ and island 2 of size $\alpha'>0$ and assume that it starts in  some $(x, y) \in [0,1]^2$, $\P$-a.s.. As we did for the seed bank model, it can be shown that the sequence $(X^{c_{\kappa},\alpha'_{\kappa}}(t), Y^{c_{\kappa},\alpha'_{\kappa}}(t))_{t \geq 0}$ will convergen to a Markovian degenerate limit coinciding in distribution with a Markov process with generator
\begin{align*}
  \bar A f(x,y) 	&= y(f(1,y)-f(0,y))\1_{\{0\}}(x) + (1-y)(f(0,y)-f(1,y))\1_{\{1\}}(x)\notag \\
			& \qquad + K(x-y)\frac{\partial f}{\partial y} (x,y) + \frac{1}{2}x(1-x)\frac{\partial^2}{\partial x^2} f(x,y)
\end{align*}
whenever started in the smaller state-space $\{0,1\}\times[0,1]$.

\end{rem}

{\bf Acknowledgement.} JB and MWB were supported by DFG Priority Programme 1590 ``Probabilistic Structures in Evolution'', project BL 1105/5-1, EB and AGC by the Berlin Mathematical School.

\bibliographystyle{alpha}

\bibliography{Bib_WFdiffusion}

\end{document}